\documentclass[12pt,reqno]{article}
\usepackage{graphicx}
\usepackage{amssymb}
\usepackage{amsmath}
\usepackage{amsthm}
\usepackage{amsfonts}
\usepackage{url}
\usepackage{hyperref}
\usepackage{breakurl}

\newcommand{\comment}[1]{}

\newcommand{\raisedot}{\raisebox{2pt}{$.$}}
\newcommand{\Dbar}{{\mathcal R}}
\newcommand{\Tabcompare}{{Table~1}}	

\newcommand{\R}{{\mathbb R}}

\newcommand{\Z}{{\mathbb Z}}

\newtheorem{theorem}{Theorem}
\newtheorem{corollary}[theorem]{Corollary} 
\newtheorem{lemma}[theorem]{Lemma}	   

\newtheorem{definition}[theorem]{Definition}
\newtheorem{remark}[theorem]{Remark}

\begin{document}
\bibliographystyle{plain}
\title{~\\[-40pt] 
General lower bounds on maximal determinants of binary matrices}
\author{Richard P. Brent\\
Australian National University\\
Canberra, ACT 0200,
Australia\\
\href{mailto:maxdet@rpbrent.com}{\tt maxdet@rpbrent.com}
\and
Judy-anne H. Osborn\\
The University of Newcastle\\
Callaghan, NSW 2308,
Australia\\
\href{mailto:Judy-anne.Osborn@newcastle.edu.au}%
{\tt Judy-anne.Osborn@newcastle.edu.au}
}

\date{\small In memory of Warwick Richard de Launey 1958--2010}

\maketitle
\thispagestyle{empty}                   

\begin{abstract}
We prove general
lower bounds on the maximal determinant of
\hbox{$n \times n$} $\{+1,-1\}$-matrices,
both with and without the assumption of the Hadamard conjecture.
Our bounds improve on earlier results of de Launey and Levin (2010)
and, for certain congruence classes of $n \bmod 4$, the results of
Koukouvinos, Mitrouli and Seberry (2000).  
In an Appendix we give a new proof, using Jacobi's determinant identity,
of a result of Sz\"oll\H{o}si (2010) on minors of Hadamard matrices.
\end{abstract}

\section{Introduction}		\label{sec:intro}

For $n \ge 1$, let
$D(n)$ denote the maximum determinant attainable by an $n \times n$
$\{+1,-1\}$-matrix.  
There are several well-known {upper} bounds on $D(n)$,
such as Hadamard's original bound~\cite{Hadamard}
$D(n) \le n^{n/2}$, which applies for all positive integers $n$, and bounds due to
Ehlich~\cite{Ehlich64a,Ehlich64b}, 
Barba~\cite{Barba33}, and
Wojtas~\cite{Wojtas64}, which are stronger but apply only to
certain congruence classes of $n \bmod 4$. 

In this paper we give new lower bounds on $D(n)$, improving
in certain cases on earlier
results of Cohn~\cite{Cohn63}, Clements and Lindstr\"om~\cite{CL65},
Koukouvinos, Mitrouli and Seberry~\cite{KMS00},
and de Launey and Levin~\cite{LL}.

Since $D(n)$ is a rapidly increasing function of $n$, it is
convenient to normalize by the Hadamard bound.
Thus, we define $\Dbar(n) := D(n)/n^{n/2}$ and express our bounds in
terms of $\Dbar(n)$.  Hadamard's inequality
becomes $\Dbar(n) \le 1$.

We consider square $\{+1,-1\}$-matrices. The \emph{order} is the number
of rows (or columns) of such a matrix.
A $\{+1,-1\}$-matrix $H$ with $|\det H| = n^{n/2}$ is called a 
\emph{Hadamard matrix}. A Hadamard matrix has order $1$, $2$, or a multiple
of $4$; the \emph{Hadamard conjecture} is that every positive multiple 
of~$4$ is the order of a Hadamard matrix. It is known from~\cite{KT} 
that every positive multiple of $4$ up to and including $664$ is 
the 
order of a Hadamard matrix.

Our technique for obtaining lower bounds on $D(n)$ is to consider a Hada\-mard
matrix $H$ of order 
$h$ as close as possible to $n$. If $h > n$ we 
consider minors of order $n$ in $H$, much as was done by de Launey and
Levin~\cite{LL}, although the details differ as we use a theorem of
Sz\"oll\H{o}si~\cite{Szollosi10} instead of the probabilistic
approach of~\cite{LL}.
If $h < n$ we construct a matrix of
order $n$ with large determinant having $H$ as a submatrix. By combining 
both ideas, we improve on the bounds that are attainable using either
idea separately.

The distance $\delta(n) = |h-n|$ of $n$ from the (closest) order $h$ of a
Hadamard matrix can be bounded by the \emph{prime gap} function $\lambda(x)$ which
bounds the maximum distance between successive primes $p_i, p_{i+1}$ with
$p_i \le x$. Thus, we can use bounds on $\lambda(x)$, such as the
theorem of Baker, Harman and Pintz~\cite{BHP}, to obtain unconditional lower
bounds on $D(n)$. Unfortunately, the known bounds on $\lambda(x)$
are much weaker than what is conjectured to be true.

We give unconditional lower bounds on $D(n)$ and $\Dbar(n)$ in
\S\ref{sec:bounds}.
Theorem~\ref{thm:unconditional1}
implies that $\Dbar(n) \ge n^{-\delta(n)/2}$.

In \S\ref{sec:H_bounds} we give stronger lower bounds on the
assumption of the Hadamard conjecture.
Theorem~\ref{thm:conditional1} improves (for large~$n$) on
the bounds of Koukouvinos, Mitrouli and Seberry~\cite{KMS00}
in the cases $n \bmod 4 \in \{1,2\}$.

On the assumption of the Hadamard conjecture,
the relative gap between the (Hadamard) upper bound and the lower bound is
of order at most $n^{1/2}$.
More precisely, our Corollary~\ref{cor:improved}
gives $\Dbar(n) \ge (3n)^{-1/2}$.
This improves on the lower bound of de Launey and Levin~\cite{LL}, 
who obtained $\Dbar(n) \ge cn^{-3/2}$
for some constant~$c>0$.
A comparison of our bounds with earlier results
is given in \S\ref{sec:comparison} (see also Remark~\ref{remark:LL}
in \S\ref{sec:bounds}).

Our lower bound results are weaker than what is conjectured to be
true. Numerical evidence for $n \le 120$ supports a conjecture of Rokicki
\emph{et al}~\cite{Orrick-prog} that $\Dbar(n) \ge 1/2$.
In~\S\ref{sec:H_bounds} we come close to this conjecture (on the assumption of
the Hadamard conjecture) for five of the eight congruence classes
of $n \bmod 8$.

\subsubsection*{Notation}		

The positive integers are denoted by $\Z^{+}$, and the reals by $\R$.
The notations $f \ll g$ and $g \gg f$ mean the same as $f = O(g)$.

For $n\in\Z^{+}$, ${\cal H}_n$ denotes the
set of Hadamard matrices of order $n$,
and ${\cal H} := \{n\in\Z^{+}\;|\; {\cal H}_n \ne \emptyset\}$.
The elements of $\cal H$ in increasing order form the 
sequence $(n_i)_{i \ge 1}$ of all possible
orders of Hadamard matrices ($n_1 = 1$, $n_2 = 2$, $n_3=4$, $n_4 =
8$, $n_5 = 12, \ldots$). 
The distance of $n$ from a Hadamard order is
\begin{equation} \label{def:delta}
\delta(n) := \min_{h \in \cal H} |n-h|.
\end{equation}

\pagebreak[3]
The primes are denoted by $(p_i)_{i \ge 1}$ with $p_1 = 2, p_2 = 3$, etc.
The \emph{prime gap} function $\lambda:\R \to \Z$ is 
\[\lambda(x) := \max\;\{p_{i+1}-p_i\;|\; p_i \le x\} \cup \{0\}.\]
By analogy, we define the \emph{Hadamard gap} function $\gamma:\R\to\Z$ to be
\[\gamma(x) := \max\;\{n_{i+1}-n_i\;|\; n_i \le x\} \cup \{0\}.\]

Finally, $\beta_n$ denotes the well-known mapping from $\{+1,-1\}$-matrices of 
order $n > 1$ to $\{0,1\}$-matrices of order $n-1$,
such that \[|\det(A)| = 2^{n-1}|\det\beta_n(A)|.\]

\section{Unconditional lower bounds on 
 {\it D}({\it n})} 
 \label{sec:bounds}

The connection between the prime gap function $\lambda$
and the Hadamard gap function $\gamma$
is given by the following lemma.
\begin{lemma}	\label{lemma:gamma_ineq}
For $n \ge 8$, we have $\gamma(n) \le 2\lambda(n/2-1)$.
\end{lemma}
\begin{proof}
If $p$ is an odd prime, then $n = 2(p+1) \in \cal H$. This follows
from the second Paley construction~\cite{Paley33} if $p \equiv 1 \pmod 4$,
or from the first Paley construction followed by the Sylvester construction
if $p \equiv 3 \pmod 4$.  Thus, if $p_i$, $p_{i+1}$ are consecutive odd primes,
then $n_j = 2(p_i + 1) \in \cal H$,\linebreak
$n_k = 2(p_{i+1} + 1) \in \cal H$, and $k > j$.
The result now follows from the definitions of the two gap functions.
\end{proof}

\begin{remark}
{\rm
De Launey and Gordon~\cite{LG} have shown that the
sequence of Hadamard orders $(n_i)$ is asymptotically denser
than the sequence of primes. Even if we consider only the Paley and
Sylvester constructions and Kronecker products arising from
them~\cite{Agaian}, 
we can frequently find Hadamard matrices whose orders
lie in the interior of the interval
$(2(p_i+1),2(p_{i+1}+1))$ defined by 
a large prime gap. 
}
\end{remark}

\begin{corollary}	\label{cor:delta_ineq}
For $n \ge 8$, we have $\delta(n) \le \lambda(n/2-1)$.
\end{corollary}
\begin{proof}
By the definition of $\delta(n)$ we have $\delta(n) \le \gamma(n)/2$,
so the result follows from Lemma~\ref{lemma:gamma_ineq}.
\end{proof}

Lemma~\ref{lemma:ineq1} gives an inequality that is often useful.

\begin{lemma} \label{lemma:ineq1}
If $\alpha \in \R$, $n \in \Z$, and $n > |\alpha| > 0$, then
\[\frac{(n-\alpha)^{n-\alpha}}{n^n} >
\left(\frac{1}{ne}\right)^\alpha\,\raisedot\]
\end{lemma}
\begin{proof}
Taking logarithms, and writing $x = \alpha/n$,
the inequality reduces to
\[(1-x)\ln(1-x) + x > 0,\]
or equivalently (since $0 < |x| < 1$)
\[\frac{x^2}{1\cdot 2} + \frac{x^3}{2\cdot 3} + \frac{x^4}{3\cdot 4} +
\cdots > 0.\]
This is clear if $x>0$, and also if $x < 0$ because then the terms alternate
in sign and decrease in magnitude.
\end{proof}

Recently Sz\"oll\H{o}si~\cite[Proposition 5.5]{Szollosi10} established
an elegant correspondence between the minors of order $n$ and of order $h-n$
of a Hadamard matrix of order $h$.  His result applies to complex Hadamard
matrices, of which $\{+1,-1\}$-Hadamard matrices are a special case.
More precisely, if $d+n = h$, $0 < d < h$, then for each minor
of order $d$ and value $\Delta$ there corresponds
a minor of order $n$ and value $\pm h^{h/2 - d}\Delta$. 
Previously, only a few special cases (for small~$d$ or $n$, see for
example~\cite{DP88,KMS01,SXKM03,Sharpe07}) were known. We
note that Sz\"oll\H{o}si's crucial Lemma~5.7 follows easily from
Jacobi's determinant identity~\cite{BS83,Gantmacher,Jacobi},
although Sz\"oll\H{o}si gives a different proof.\footnote{In the Appendix
we give a proof of Sz\"oll\H{o}si's Lemma~5.7 using Jacobi's identity.}

\begin{lemma} \label{lemma:minor}
Suppose $0 < n < h$ and $h \in \cal H$.
Then $D(n) \ge 2^{d-1}h^{h/2 - d}$, where $d = h-n$.
\end{lemma}
\begin{proof}
Let $H \in {\cal H}_h$ be a Hadamard matrix of order~$h$,
and let $M$ be any $n\times n$ submatrix of $H$ (so $M$ does not necessarily
have contiguous rows or columns in $H$).  Let $M'$ be the
$d\times d$ submatrix consisting of the intersection of the 
complementary set of rows and columns of $H$.  Some
such $M'$ must be nonsingular, else we could prove,
using Laplace's expansion of the determinant and induction on $n$,
that $\det(H)=0$, contradicting the assumption that
$H$ is a Hadamard matrix.  
Thus, without loss of generality, $\det(M') \ne 0$.
Since $M'$ is a $\{\pm1\}$-matrix, we must have
$|\det(M')| \ge 2^{d-1}$.
By~Sz\"oll\H{o}si's theorem,
$|\det(M)| = h^{h/2-d}|\det(M')| \ge 2^{d-1}h^{h/2-d}$.
\end{proof}
\begin{remark}
{\rm
We could improve Lemma~\ref{lemma:minor} for large $d$ 
by using the fact that,
from a result of de Launey and Levin~\cite[proof of Prop.~5.1]{LL}, 
there exists $M'$ with
$|\det(M')| \ge (d!)^{1/2}$, which is asymptotically larger than the bound
$|\det(M')| \ge 2^{d-1}$ that we used in our proof.
However, in our application of the
lemma, $h \gg d$, so it is the power of $h$ in the bound that is
significant.
}
\end{remark}

\begin{lemma} \label{lemma:major}
Suppose $0 < h < n$ and $h \in \cal H$. Then
$D(n) \ge 2^{n-h}h^{h/2}$.
\end{lemma}
\begin{proof}
The case $h=1$ is trivial, so suppose that $h > 1$.
Let $H \in {\cal H}_h$ be a Hadamard matrix of order~$h$, so
$H$ has determinant $\pm h^{h/2}$ and the corresponding $\{0,1\}$-matrix
$\beta_h(H)$ has determinant $\pm 2^{1-h}h^{h/2}$. We can
construct a $\{0,1\}$-matrix $A$ of order $n-1$ and the same determinant
as $\beta_h(H)$ by adding a border of $n-h$ rows and columns (all zero
except for the diagonal entries). Now construct a $\{+1,-1\}$-matrix 
$B \in \beta_n^{(-1)}(A)$ by applying the standard mapping from 
$\{0,1\}$-matrices to $\{+1,-1\}$-matrices. We have
$|\det(B)| = 2^{n-1}|\det(A)| = 2^{n-h}h^{h/2}$.
\end{proof}

\begin{lemma} \label{lemma:n_delta_bound}
Let $n \in \Z^{+}$ and $\delta = \delta(n)$ be defined by \eqref{def:delta}.
Then $n \ge 3\delta$.
\end{lemma}
\begin{proof}
The interval $[2n/3,4n/3)$ contains a unique power of two, say $h$.
By the Sylvester construction, $h \in {\cal H}$.  However,
$|n-h| \le n/3$, so $\delta \le n/3$.
\end{proof}

\begin{theorem} \label{thm:unconditional1}
If $n\in\Z^{+}$ and
$\delta = \min_{h\in\cal H}|n-h|$, then
\begin{equation} \label{eq:uncondbd}
\Dbar(n) \ge \left(\frac{4}{ne}\right)^{\delta/2}.
\end{equation}
\end{theorem}
\begin{proof}
By the definition of $\delta$, there exists a Hadamard matrix $H$
of order $h = n \pm \delta$. If $\delta = 0$ the result is trivial,
so suppose $\delta \ge 1$.
We consider two cases.
First suppose that $h = n + \delta$.  Applying Lemma~\ref{lemma:minor},
we have \[D(n) \ge 2^{\delta-1}h^{h/2 - \delta} \ge h^{h/2 - \delta}.\]
Now, applying Lemma~\ref{lemma:ineq1} with $\alpha = -\delta$ gives
\[\Dbar(n) 
  \ge \frac{h^{h/2-\delta}}{n^{n/2}}
  = \frac{(n+\delta)^{(n+\delta)/2}}{n^{n/2}}(n+\delta)^{-\delta}
  \ge \left(\frac{ne}{(n+\delta)^2}\right)^{\delta/2}.
\]
By Lemma~\ref{lemma:n_delta_bound} we have 
$\delta/n \le 1/3 < (e/2 - 1)$, from which it is easy to verify that
$ne/(n+\delta)^2 > 4/(ne)$. 
The inequality \eqref{eq:uncondbd} follows.

Now suppose that $h = n - \delta$. 
From Lemma~\ref{lemma:major}
we have $D(n) \ge 2^\delta h^{h/2}$. 
Using Lemma~\ref{lemma:ineq1} with $\alpha = \delta$, we have
\[\Dbar(n) > 2^{\delta}\left(\frac{1}{ne}\right)^{\delta/2}
 = \left(\frac{4}{ne}\right)^{\delta/2}.\]
Thus, in all cases we have established the desired lower bound on $\Dbar(n)$. 
\end{proof}
\begin{remark}	\label{remark:LL}
{\rm
Consider $n>4$ in the interval $(n_i, n_{i+1})$ between two consecutive
Hadamard orders,
and write $\Delta := (n_{i+1}-n_i)/2 \ge 2$.
De Launey and Levin~\cite[Theorem 3]{LL} 
take $d = n_{i+1}-n \le 2\Delta-1$
and give (in our notation) the bound $\Dbar(n) \ge n^{-d/2}$.
In contrast, our bound is $\Dbar(n) \ge n^{-\delta/2}$, where
$\delta \le \Delta$.
Note that $\max(d)+1 = 2\max(\delta) = 2\Delta$.
In the worst case, the bound of de Launey and Levin is
$n^{-(2\Delta-1)/2}$, whereas the worst case for our bound is
$n^{-\Delta/2}$.  Thus, we almost halve the exponent
of $n$ in the worst-case bound.
The reason for the difference is that de Launey and Levin always 
take a Hadamard
matrix with order $h = n_{i+1} > n$, whereas we take $h = n_i < n$ and use
Lemma~\ref{lemma:major} if that gives a sharper bound.
}
\end{remark}

\begin{corollary} \label{cor:uncond}
For $n\in\Z$, 
$n \ge 4$, 
\[\Dbar(n) \ge \left(\frac{4}{ne}\right)^{\lambda(n/2)/2},\]
where $\lambda(n)$ is the prime gap function defined above.
\end{corollary}
\begin{proof}
For $n \ge 8$ this
follows from Theorem~\ref{thm:unconditional1}, using
Corollary~\ref{cor:delta_ineq}. It is easy
to check that the inequality holds
for $4 \le n < 8$ by using the known values of $D(n)$ listed
in~\cite{Orrick-www}. 
\end{proof}
\begin{remark} \label{remark:Hoheisel}
{\rm
In the literature
there are many inequalities for $\lambda(n)$,
see for example Hoheisel~\cite{Hoheisel} or Huxley~\cite{Huxley}.
The best result so far  seems to be that of 
Baker, Harman and Pintz~\cite{BHP}, who proved that
$\lambda(n) \le n^{21/40}$ for $n \ge n_0$, where $n_0$ is a sufficiently
large (effectively computable) constant.
Assuming the Riemann hypothesis, Cram\'er~\cite{Cramer} proved
that $\lambda(n) = O(n^{1/2}\log n)$.
``Cram\'er's conjecture'' is that
$\lambda(n) = O((\log n)^2)$, and numerical 
computations~\cite{Nicely-gaps,Shanks64,Silva}
provide some evidence for this conjecture.
For a discussion of other relevant results on prime gaps,
see~\cite[\S1]{LL}.
}
\end{remark}

\pagebreak[3]
\begin{corollary} \label{cor:asymptotics}
If $n\in N$, then
\[0 \le {n\ln n} - 2\ln D(n) = O(n^{21/40}\ln n)
  \;\text{\rm\ as }\; n \to \infty.\]
\end{corollary}
\begin{proof}
The result follows from Corollary~\ref{cor:uncond} and the theorem
of Baker, Harman and Pintz~\cite{BHP}.
\end{proof}

\section{Conditional lower bounds on 
 {\it D}({\it n})} 
 \label{sec:H_bounds}

In this section we assume the Hadamard conjecture
and give lower bounds on $D(n)$ that are sharper than
the unconditional bounds of~\S\ref{sec:bounds}.

The idea of the proof of Theorem~\ref{thm:conditional1}
is similar to that of Theorem~\ref{thm:unconditional1}~-- 
we use a Hadamard matrix of slightly
smaller or larger order to bound $D(n)$ when $n \not\equiv 0 \pmod 4$. In each
case, we choose whichever construction gives the sharper bound.
First we make a definition and state two well-known lemmas.
\begin{definition}
Let $A$ be a $\{\pm1\}$-matrix. 
The \emph{excess} of $A$ is $\sigma(A) := \sum_{i,j}a_{i,j}$.
If $n \in {\cal H}$, then 
$\sigma(n) := \max_{H \in {\cal H}_n} \sigma(H)$.
\end{definition}
The following lemma is a corollary of \cite[Theorem~1]{EM},
and gives a small improvement on Best's lower bound~\cite[Theorem~3]{Best}
$\sigma(h) \ge 2^{-1/2}h^{3/2}$.
\begin{lemma}	\label{lemma:sigma_lower}
If $4 \le h \in {\cal H}$, then 
\begin{equation*}	
\sigma(h) \ge (2/\pi)^{1/2}h^{3/2}.
\end{equation*}
\end{lemma}
The following lemma is ``well-known''~-- it 
follows from~\cite[Theorem~2]{SW}
and is also mentioned in later works such as~\cite[pg.~166]{FK}.
\begin{lemma}	\label{lemma:sigma_upper}
If $h \in {\cal H}$, then
\begin{equation*}	
D(h+1) \ge h^{h/2}\left(1 + \frac{\sigma(h)}{h}\right).
\end{equation*}
\end{lemma}
\begin{theorem}	\label{thm:conditional1}
Assume the Hadamard conjecture. For $n \in \Z^{+}$, 
we have
\begin{equation} \label{eq:three_cases}
\Dbar(n) \ge 
\begin{cases}
\left(\frac{2}{\pi e}\right)^{1/2} &
 \text{\rm if $n \equiv 1 \pmod 4$,}\\
\left(\frac{8}{\pi e^2 n}\right)^{1/2} &
 \text{\rm if $n \equiv 2 \pmod 4$,}\\
(n+1)^{(n-1)/2}/n^{n/2} \;\,\sim \left(\frac{e}{n}\right)^{1/2} &
 \text{\rm if $n \equiv 3 \pmod 4$.}
\end{cases}
\end{equation}
\end{theorem}
\begin{proof}
Since $\Dbar(1) = \Dbar(2) = 1$, the result holds for $n \in \{1,2\}$,
so we assume that $n \ge 3$.
Suppose that $4 \le h \equiv 0 \pmod 4$. We are assuming the Hadamard
conjecture, so $h \in {\cal H}$. Thus, combining
the inequalities of Lemma~\ref{lemma:sigma_lower} 
and Lemma~\ref{lemma:sigma_upper}, we have
\begin{equation} \label{eq:Dh1}
D(h+1) \ge h^{h/2}(1 + (2h/\pi)^{1/2})\,.
\end{equation}
Let $A$ be a $\{\pm1\}$-matrix of order $h+1$ with determinant
at least the right side of~\eqref{eq:Dh1}.
By the argument used in the proof of Lemma~\ref{lemma:major}, we can construct
a $\{\pm1\}$-matrix of order $h+2$ with determinant at least
$2h^{h/2}(1 + (2h/\pi)^{1/2})$ by adjoining a row and column
to $A$. 
Thus
\begin{equation} \label{eq:Dh2}
D(h+2) \ge 2h^{h/2}(1 + (2h/\pi)^{1/2})\,.
\end{equation}

To prove the first inequality in~\eqref{eq:three_cases}, put
$h = n-1$ in~\eqref{eq:Dh1} and use Lemma~\ref{lemma:ineq1} with 
$\alpha = 1$. Thus, for $1 < n \equiv 1 \pmod 4$,
\[\Dbar(n) \ge 
 \left(\frac{2}{\pi e}\right)^{1/2}\left(
 \left(1-\frac{1}{n}\right)^{1/2}\!\! +\, 
  \left(\frac{\pi}{2n}\right)^{1/2}\right)
   > \left(\frac{2}{\pi e}\right)^{1/2}\!.\]

To prove the second inequality in~\eqref{eq:three_cases}, put
$h = n-2$ in~\eqref{eq:Dh2} and use Lemma~\ref{lemma:ineq1} with 
$\alpha = 2$. Thus, for $2 < n \equiv 2 \pmod 4$,
\[\Dbar(n) \ge \left(\frac{8}{\pi e^2 n}\right)^{1/2}\left(
 \left(1 - \frac{2}{n}\right)^{1/2} \!\!+\,
  \left(\frac{\pi}{2n}\right)^{1/2}\right)
   > \left(\frac{8}{\pi e^2 n}\right)^{1/2}\!.
\]
Finally, if $n\equiv 3 \pmod 4$, then a Hadamard matrix of order $n+1$ exists.
From Lemma~\ref{lemma:minor} with $h = n+1$ we have
$D(n) \ge (n+1)^{(n-1)/2}$. 
\end{proof}
\begin{corollary} \label{cor:improved}
Assume the Hadamard conjecture. If $n \ge 1$ then
\[\Dbar(n) \ge 1/\sqrt{3n}\,.\]
\end{corollary}
\begin{proof}
For $n>2$ this follows from Theorem~\ref{thm:conditional1},
since $\pi e^2 < 24$
(in fact we could replace the constant $3$ in the statement of the
Corollary by  $\pi e^2/8 \approx 2.9017$).
The result is also true if $n \in \{1,2\}$, as then $\Dbar(n) = 1$.
\end{proof}
\begin{remark}
{\rm
The inequality~\eqref{eq:Dh1} is within a factor $\sqrt{\pi}$ of the
Barba upper bound $(2h+1)^{1/2}h^{h/2}$.
}
\end{remark}
\begin{remark}
{\rm
If $n \equiv 2 \pmod 8$, we get a lower bound
$\Dbar(n) \ge 2/(\pi e)$ by using the Sylvester
construction on a matrix of order $n/2 \equiv 1 \pmod 4$.
Thus, the remaining cases in which there is a ratio of order $n^{1/2}$
between the upper and lower bounds are
$(n \bmod 8) \in \{3,6,7\}$.
}
\end{remark}

\section{Comparison with earlier results}	\label{sec:comparison}

Since different authors use different notations, it is not always easy to
compare their lower bounds. To assist the reader in this, we briefly compare
our results with the earlier lower-bound results of Cohn~\cite{Cohn63},
Clements and Lindstr\"om~\cite{CL65}, Koukouvinos, Mitrouli and
Seberry~\cite{KMS00}, and de Launey and Levin~\cite{LL}.

Cohn~\cite[Theorem 13]{Cohn63} shows that, for any given positive
$\varepsilon$ and all sufficiently large $n$,
$D(n) \ge n^{(1/2-\varepsilon)n}$. 
This inequality is
equivalent to \[n\ln n - 2\ln D(n) \le 2\varepsilon n\ln n\,.\] 
Thus, we can express 
Cohn's result as
$\ln D(n) \sim \frac12 n\ln n$, or equivalently
\[n\ln n - 2\ln D(n) = o(n\ln n) \;\text{ as }\; n \to \infty\,.\]

Clements and Lindstr\"om~\cite[Corollary to Theorem 2]{CL65}
improved Cohn's result by showing that the $o(n\ln n)$ term could be
replaced by $O(n)$. More precisely, they obtained the bound
\[n\ln n - 2\ln D(n) \le n\ln(4/3)\,.\]

Our Corollary~\ref{cor:asymptotics} improves (at least asymptotically)
on the results of Cohn, Clements and
Lindstr\"om by showing that
\[n\ln n - 2\ln D(n) = O(n^{21/40}\ln n) \;\text{ as }\; n \to \infty\,.\]
The exponent $21/40$ here arises from a bound~\cite{BHP} on prime gaps.

Koukouvinos, Mitrouli and Seberry~\cite[Theorem~2]{KMS00} assume that
$4t = v+1$ is the order of a Hadamard matrix, and consider
orders $v$, $v-1$ and $v-2$ separately. 
They obtain lower bounds of $(4t)^{2t-1}$,
$2(4t)^{2t-2}$, and $4(4t)^{2t-3}$ respectively in these cases.

On the assumption that both $4t-4$ and $4t$
are orders of Hadamard matrices, the comparison with our
Theorem~\ref{thm:conditional1} is summarized in {\Tabcompare}.
The asymptotics all 
follow from the fact that $\lim_{n\to\infty}(1+c/n)^n = \exp(c)$.
For example, the case $n \equiv 1 \bmod 4$ corresponds to taking
minors of order $n = v-2 = 4t-3$, 
and the lower bound of Koukouvinos \emph{et al} is
\[4(4t)^{2t-3} = 4(n+3)^{(n-3)/2} 
 \sim 4e^{3/2}n^{(n-3)/2} \;\text{ as }\; n \to \infty\,.\]
From {\Tabcompare}
we see that the bounds are the same in the case $n \equiv 3
\bmod 4$, but our bounds are sharper (for sufficiently large $n$)
in the other two cases.  More precisely, our Theorem~\ref{thm:conditional1}
gives sharper bounds than Theorem~2 of~\cite{KMS00}
if $n \ge 9$ in the case $n \equiv 1 \bmod 4$,
and if $n \ge 82$ in the case $n \equiv 2 \bmod 4$.

\begin{table}[h]        
\begin{center}          
\begin{tabular}{|c|c|c|c|}
\hline
$n$ & $n \bmod 4$ & Koukouvinos \emph{et al}~\cite{KMS00} &
	Our Theorem~\ref{thm:conditional1} \\
\hline
& & & \\[-8pt]
$v-2$ & $1$	& $4(e/n)^{3/2} \approx 17.93/n^{3/2}$ &
	$(2/(\pi e))^{1/2} \approx 0.4839$ \\
& & & \\[-8pt]
$v-1$ & $2$ 	& $\;\;2e/n \approx 5.437/n$
	& $(8/(\pi e^2 n))^{1/2} \approx 0.5871/n^{1/2}$\\
& & & \\[-8pt]
$v$ & $3$	& $(e/n)^{1/2} \approx 1.649/n^{1/2}$
	& $(e/n)^{1/2} \approx 1.649/n^{1/2}$\\
\hline
\end{tabular}
\caption{Asymptotics of some lower bounds on $\Dbar(n)$}
\end{center}
\end{table}

We give two examples.
First consider $n = 13 \equiv 1 \bmod 4$.
Theorem~2 of
\cite{KMS00} (with $t=4$, $v=15$, $n=v-2$)
gives $D(13) \ge 4(4t)^{2t-3} = 4194304$,
so $\Dbar(13) \ge 0.2410$.  Our Theorem~\ref{thm:conditional1}
gives the sharper bound $\Dbar(13) \ge 0.4839$. The 
maximal determinant is
known from~\cite{Raghavarao} to be
$D(13) = 14929920$, so $\Dbar(13) \approx 0.8579$.

As a second example, consider $n = 94 \equiv 2 \bmod 4$.
Theorem~2 of
\cite{KMS00} (with $t=24$, $v=95$, $n=v-1$) gives
a lower bound $D(94) \ge 2\cdot 96^{46}$,
so $\Dbar(94) \ge 0.0560$, whereas
our Theorem~\ref{thm:conditional1} gives
$\Dbar(94) \ge 0.0605$. This bound can be improved
by a construction due to Rokicki, see \cite{Orrick-www},
but the exact value of $\Dbar(94)$ is unknown.

Our Theorem~\ref{thm:unconditional1} is more general than Theorem~2 of
Koukouvinos \emph{et~al}, as Theorem~\ref{thm:unconditional1} 
covers the cases $n\le v-3$ which occur
if the Hadamard conjecture is false and a Hadamard matrix of order $4t-4$
does not exist.
If $n = v$ or $v-1$ then Theorem~\ref{thm:unconditional1} gives
bounds of the same order
of magnitude as those of Koukouvinos \emph{et~al}
(of order $n^{-1/2}$ and $n^{-1}$ respectively),
which is to be expected as in the first half of the proof of
Theorem~\ref{thm:unconditional1} (and Lemma~\ref{lemma:minor})
we use a similar argument involving minors of a Hadamard matrix
of order $v+1$. Our bounds are slightly weaker as 
the constant $4/e$ in the inequality~\eqref{eq:uncondbd}
is not optimal in all cases.

As explained in Remark~\ref{remark:LL} of \S\ref{sec:bounds}, 
our Theorem~\ref{thm:unconditional1}
improves on Theorem~3 of de Launey and Levin~\cite{LL}
by almost halving the worst-case exponent of~$n$.

\section{Conclusion and remarks} 

Recall that $\delta(n) = |h-n|$ is the distance from a given order $n$
to the closest order $h$ of a Hadamard matrix.  We have shown
that $\Dbar(n) \ge n^{-\delta(n)/2}$ (see Theorem~\ref{thm:unconditional1}). 
On the assumption of the Hadamard
conjecture, 
this can be improved to
$\Dbar(n) \ge (3n)^{-1/2}$ (see Corollary~\ref{cor:improved}).

In view of the numerical result $\Dbar(n) \ge 1/2$ that holds for $n \le 120$
(see~\cite{Orrick-www}),
our bounds seem far from the best possible.
The best prospect of improving them may be to apply 
the probabilistic method, as was done in the case $n = h+1$
by Brown and Spencer~\cite{BS71}
and (independently) by Best~\cite{Best}.
For preliminary results in this direction, see the work in progress
at~\cite{rpb253}.

\subsubsection*{Acknowledgements}

We thank Will Orrick for his assistance in locating some of the
references, and Warren Smith for pointing out the connection 
between Jacobi's identity and Sz\"oll\H{o}si's theorem.
We also thank the referee, whose comments helped significantly to
clarify the exposition.

\pagebreak[4]

\section{Appendix: Proof of Sz\"oll\H{o}si's Lemma 5.7}

Here we give a short proof of Lemma 5.7 of Sz\"oll\H{o}si~\cite{Szollosi10},
using Jacobi's determinant identity~\cite{BS83,Gantmacher,Jacobi}.

\begin{lemma}[Sz\"oll\H{o}si]	\label{lemma:Sz}
Given any unitary matrix 
\[U = \left[\begin{matrix} A & B\\ C & D\\ \end{matrix}\right] \]
with blocks $A$, $B$, $C$, $D$, where $A$ and $D$ are square matrices
not necessarily of the same size, then we have
$|\det(A)| = |\det(D)|$.
\end{lemma}
\begin{proof} 
Since $U$ is unitary, we have
\[
U^{-1} = U^{*} = \left[\begin{matrix} A^{*} & C^{*}\\ 
 B^{*} & D^{*}\\ \end{matrix}\right]\,,
\]
where ``$*$'' denotes the complex conjugate transpose.
Thus, from Jacobi's identity,
\[\det(A) = \det(U)\det(D^{*})\,.\]
Taking absolute values and using $|\det(U)| = 1$, we obtain
\[|\det(A)| = |\det(D^{*})| = |\det(D)|\,.\]
\end{proof}

\begin{remark}	\label{remark:Sz}
{\rm
If $H$ is a Hadamard matrix of order $h$, we can apply
Lemma~\ref{lemma:Sz} to
$U := h^{-1/2}H$ which is a unitary matrix.  Thus, if
$H$ is written in block form as
\[H = \left[\begin{matrix} A & B\\ C & D\\ \end{matrix}\right]\,, \]
where $A$ is $n\times n$ and $D$ is $d\times d$, we have
\[|\det(h^{-1/2}A)| = |\det(h^{-1/2}D)|\,.\] 
Equivalently, since $h=n+d$,
\[|\det(A)| = 
  h^{h/2-d}|\det(D)|\,,\]
which is the result that we use in \S\ref{sec:bounds}.
}
\end{remark}

\pagebreak[4]

\end{document}